\documentclass[preprint,12pt]{elsarticle}

\usepackage{amssymb}
\usepackage{amsthm}
\usepackage{amsmath}
\usepackage{tikz}
\usepackage{multirow}
\usepackage{tabu}
\usepackage{tabularx}
\usepackage{float}
\usepackage[utf8]{inputenc} % usually not needed (loaded by default)
\usepackage[T1]{fontenc}
\usetikzlibrary{positioning}

\usepackage{lineno}

\newdefinition{remark}{Remark}
\newdefinition{example}{Example}
\newtheorem{proposition}{Proposition}
\newtheorem{lemma}{Lemma}
\newtheorem{theorem}{Theorem}
\newtheorem{corollary}{Corollary}

\newcolumntype{C}{>{\centering\arraybackslash}X}

\newcolumntype{P}[1]{>{\centering\arraybackslash}p{#1}}

\makeatletter
\def\ps@pprintTitle{%
   \let\@oddhead\@empty
   \let\@evenhead\@empty
   \let\@oddfoot\@empty
   \let\@evenfoot\@oddfoot
}
\makeatother

\begin{document}

\begin{frontmatter}

\title{Notes on the lattice of fuzzy rough sets with crisp reference sets\tnoteref{fr_title}}

\author[1]{Dávid Gégény\corref{cor_gd}\fnref{fund}}
\ead{matgd@uni-miskolc.hu}
\ead[url]{http://www.uni-miskolc.hu/~matgd/}

\author[2]{László Kovács\fnref{fund}}
\ead{kovacs@iit.uni-miskolc.hu}
\ead[url]{https://www.iit.uni-miskolc.hu/munkatarsak/kovacs-laszlo.html}

\author[1]{Sándor Radeleczki\fnref{fund}}
\ead{matradi@uni-miskolc.hu}
\ead[url]{http://www.uni-miskolc.hu/~matradi/}

\address[1]{Institute of Mathematics, University of Miskolc, 3515 Miskolc-Egyetemváros, Hungary}

\address[2]{Department of Information Technology, University of Miskolc, 3515 Miskolc-Egyetemváros, Hungary}

\cortext[cor_gd]{Corresponding author}

\fntext[fund]{This work was carried out as part of the grant EFOP-3.6.1-16-00011: "Younger and Renewing University - Innovative Knowledge City - intelligent specialization", in the framework of the Széchenyi 2020 program. The realization of this project is supported by the European Union, co-financed by the European Social Fund.}

\begin{abstract}
Since the theory of rough sets was introduced by Zdzislaw Pawlak, several approaches have been proposed to combine rough set theory with fuzzy set theory. In this paper, we examine one of these approaches, namely fuzzy rough sets with crisp reference sets, from a lattice-theoretic point of view. We connect the lower and upper approximations of a fuzzy relation $R$ to the approximations of the core and support of $R$. We also show that the lattice of fuzzy rough sets corresponding to a fuzzy equivalence relation $R$ and the crisp subsets of its universe is isomorphic to the lattice of rough sets for the (crisp) equivalence relation $E$, where $E$ is the core of $R$. We establish a connection between the exact (fuzzy) sets of $R$ and the exact (crisp) sets of the support of $R$.

\end{abstract}

%%Research highlights
%\begin{highlights}

%\item The lattice of fuzzy rough sets is a completely %%distributive double Stone algebra

%\item Crisp sets' approximations are induced by the core %and %support of fuzzy approximations

%\item Defining exact fuzzy sets we show that they are equivalent to exact crisp sets

%\end{highlights}

\begin{keyword}

fuzzy rough set \sep lower and upper approximation \sep fuzzy equivalence \sep uncertain knowledge \sep regular double Stone lattice \sep dually well-ordered set 

\MSC[2010] 94D05 \sep 68T37 \sep 06B15

\end{keyword}

\end{frontmatter}

%% main text
\section{Introduction}
\label{s_intro}

The notion of fuzzy sets and rough sets both extend the concept of traditional (crisp) sets by incorporating that our knowledge may be uncertain or incomplete. However, these approaches address the problem of imperfect information in a different way.

Rough sets were introduced by Zdzislaw Pawlak \cite{Pawlak1}, and they use the lower and upper approximations of a (crisp) set based on the indiscernibility relation of the elements. Given a reference set $A$ in a universe $U$ and an equivalence relation $R \subseteq U \times U$, the \emph{lower approximation} of the set $A$ is
\[ A_R = \{ x \in U\ |\ [x]_R \subseteq A \} \]

\noindent and the \emph{upper approximation} of $A$ is
\[ A^R = \{ x \in U\ |\ [x]_R \cap A \neq \emptyset \} ,\]

\noindent where $[x]_R$ is the $R$-equivalence class of an element $x$. The pair $ (A_R, A^R) $ is called the \emph{rough set} corresponding to the reference set $A$ and $(U, R)$ is called an \emph{approximation space}. The rough sets corresponding to this approximation space $(U, R)$ can be ordered with respect to the component-wise inclusion, and they form a complete lattice with several particular properties, denoted by RS$(U, R)$, see e.g. \cite{Po88}, \cite{Com} and \cite{D}.

\medskip

The theory of fuzzy sets was introduced by Lotfi Zadeh \cite{Zadeh1}.  A fuzzy set $A$ is defined by a membership function $\mu_A: U \longrightarrow [0, 1]$. The membership degree 0 means that the element is certainly not a member of the set $A$, and the membership degree 1 means that the element is certainly in the set.

One of the pioneer works to analyze the relationship between the two main theories can be found in \cite{Pawlak2}, where the author had shown that there are significant differences between these concepts. The first approach to integrate the two main theories relates to the work of Dubois and Prade \cite{DP}. The proposed lower and upper approximations for fuzzy sets are defined using the t-norm Min and its dual co-norm Max. Using the symbolic notation introduced by Yao in \cite{Yao1}, the \textit{fuzzy rough set} of a fuzzy set $\Gamma$ is defined with
\[
\mu_{\underline{apr}_\mathcal{R}(\Gamma)}(x) = \text{inf} \{ \text{max} [ \mu_{\Gamma}(y), 1 - \mu_{\mathcal{R}}(x, y)] \mid y \in U \},
\]
\[
\mu_{\overline{apr}_\mathcal{R}(\Gamma)}(x) = \text{sup} \{ \text{min} [ \mu_{\Gamma}(y), \mu_{\mathcal{R}}(x, y)] \mid y \in U \},
\]

\noindent where $U$ denotes the universe set and $\mathcal{R}$ is the symbol for a fuzzy similarity relation. As the definition shows \textit{fuzzy rough sets} are rough sets having fuzzy sets as lower and upper approximations attached to a fuzzy reference set. As crisp sets are special cases of fuzzy sets (having two-valued membership functions), the given definition can also be used to construct fuzzy rough sets for crisp sets. A comparison of the two approaches can also be found in \cite{JKR}.

Beside some other generalization approaches like Nanda and Majumdar \cite{Nanda}, we can also find some different proposals for integration. The work in \cite{Coker} had shown that fuzzy rough sets are, indeed, intuitionistic L-fuzzy sets developed by Atanassov \cite{Atanassov}. More general framework can be obtained under fuzzy environment based on fuzzy similarity relations defined by t-norms, see e.g. \cite{BJR} or \cite{MY}. In \cite{BJR}, the upper and lower approximations of a fuzzy subset with respect to an indistinguishability operator are studied, and their relations with fuzzy rough sets are pointed out. In \cite{MY}, an axiomatic approach is developed; using fuzzy similarity relations defined by a t-norm, the definition of the upper and lower approximation operator in case of fuzzy rough sets is generalized based on some axiomatic properties.

\medskip

The integration proposal of Yao \cite{Yao1} is based on the consideration that a fuzzy set can be represented by a family of crisp sets using its $\alpha$-level sets, whereas a rough set can be represented by three crisp sets. Yao has analyzed the relationship between the rough fuzzy set and fuzzy rough set models and proved that rough fuzzy sets are special cases of fuzzy rough sets as defined by Dubois and Prade. Another conclusion of \cite{Yao1} is that the membership functions of rough sets, rough fuzzy sets, and fuzzy rough sets can be computed uniformly using the same scheme:

\[
    \mu_{\underline{apr}_{\Gamma}( \Delta )}(x) := \text{inf} \left\{ \text{max} \left[ \mu_{\Delta}(y), 1-\mu_{\Gamma}(x, y) \right] \mid y\in U  \right\},
\]

\[
    \mu_{\overline{apr}_{\Gamma}( \Delta )}(x) := \text{sup} \left\{ \text{min} \left[ \mu_{\Delta}(y), \mu_{\Gamma}(x, y) \right] \mid y\in U  \right\},
\]

\medskip

\noindent where $\Gamma$ is a variable that takes either an equivalence relation or a fuzzy similarity relation as its value, and $\Delta$ is a variable that takes either a crisp set or a fuzzy set as its value. The properties of the general case that uses fuzzy reference sets in a fuzzy approximation space defined by a t-norm are also examined in \cite{Cornelis1}, where an application in query refinement is also presented.

\medskip

The main application area of the fuzzy rough set theory relates to optimisation of knowledge engineering algorithms. Regarding the data preprocessing phase, the fuzzy rough set models are used mainly for attribute reduction \cite{DJQX}, \cite{KLI}. The main benefit of this approach is that fuzzy-rough feature extraction preserves the meaning, the semantics of the selected features after elimination of the redundant attributes. The FRFS method works with discovering dependencies between the elements of the attribute set. The fuzzy rough set model can also be used for general data mining operations, like clustering or classification in the case of uncertain input domains \cite{YHL}.

 The main focus of this paper is on fuzzy rough sets, using crisp sets as reference sets in a fuzzy approximation space. Fuzzy rough sets with crisp reference sets are important modelling tools in machine learning applications, like in natural language processing, where the reference sets contain crisp valued feature vectors and we construct fuzzy concept categories corresponding to them (see e.g. \cite{Yangsheng} and \cite{Gupta}). Our aim is to examine the lattice-theoretical properties of fuzzy rough sets and to draw a comparison study to traditional rough sets. We show that in case of crisp reference sets, the lattice of fuzzy rough sets corresponding to a fuzzy equivalence relation $R$ is isomorphic to the lattice of rough sets for the (crisp) equivalence relation $E$, where $E$ is the core of $R$, and this is a much investigated structure in the literature.

\medskip

Let $(U, R)$ be a fuzzy approximation space, where $U$ is the universe and $R$ is a fuzzy equivalence relation defined by a mapping $\mu_R:U^2 \longrightarrow [0, 1]$. A fuzzy equivalence relation is a reflexive, symmetric and transitive fuzzy relation. As we are considering fuzzy relations, \emph{reflexive property} means that $ \mu_R(x, x)=1 $ for every $x \in U$ and \emph{symmetry} means that $ \mu_R(x, y) = \mu_R(y, x) $ for every $x,y \in U$. Initially, a fuzzy relation $R$  was called \emph{transitive} if min$(\mu_R(x, y), \mu_R(y, z)) \leq \mu_R(x, z)$, for all $ x, y, z \in U $ \cite{Z71}. Later this notion was
generalized by using the
notion of a t-norm (see e.g. \cite{V85}). \emph{A triangular norm} ${}\mathcal{T}$ (\emph{%
t-norm} for short) is an increasing commutative and associative mapping $\mathcal{T\colon }[0,1]^{2}\longrightarrow \lbrack 0,1]$
satisfying $\mathcal{T}(1,x)=\mathcal{T}(x,1)=x$, for all $x\in \lbrack 0,1]$%
. The t-norm $\mathcal{T}$\emph{\ is }called\emph{\ positive }(see e.g\emph{%
. }\cite{BA08})\emph{, }if\emph{\ }$\mathcal{T}(x,y)>0$, whenever $x,y>0$. We say
that a fuzzy relation $\mu _{R}\colon U^{2}\longrightarrow \lbrack 0,1]$ is $%
\mathcal{T}$\emph{-transitive}, if

\[
\mathcal{T}(\mu _{R}(x,y),\mu _{R}(y,z))\leq \mu _{R}(x,z)\text{, for all }%
x,y,z\in U.
\]

\noindent A reflexive, symmetric and $\mathcal{T}$-transitive\emph{\ }fuzzy
relation\emph{\ }$R$ is called a $\mathcal{T}$\emph{-equivalence,} or a \emph{fuzzy
}$\mathcal{T}$\emph{-similarity} \emph{relation. }It is well-known that $%
\mathcal{T}(x,y)=\ $min$(x,y)$, $x,y\in \lbrack 0,1]$ is a positive t-norm
corresponding to the previous notion of transitivity.

Now, let $A \subseteq U $ be a crisp set. A \emph{fuzzy rough set with reference set $A$} is defined as a pair of two fuzzy sets corresponding to $A$ \cite{Yao1}. The lower approximation of $A$ is given by the membership function 
\[ \mu_{[A]_R} (x)=\text{inf}\{ 1 - \mu_R(x, y)\ \mid \ y \notin A\},\]
\noindent and the upper approximation of $A$ is given by the membership function
\[ \mu_{[A]^R} (x)=\text{sup}\{ \mu_R(x, y)\ \mid \ y \in A\}. \]
\noindent It is easy to check that $\mu_{\lbrack\emptyset]_{R}}=\mu_{\lbrack
\emptyset]^{R}}=\mathbf{0}$, where $\mathbf{0}$ denotes the constant $0$
mapping on $U$, and $\mu_{\lbrack U]_{R}}=\mu_{\lbrack U]^{R}}=\mathbf{1}$,
where $\mathbf{1}$ stands for the constant $1$ mapping on $U$. (Notice that
sup$\ \emptyset=0$, inf$\ \emptyset=1$, and $\mu_{R}(x,x)=1$, for any $x\in U$.) The fuzzy rough set corresponding to the crisp set $A$ is the pair $(\mu_{[A]_R}, \mu_{[A]^R})$. We know that the set of all rough sets in approximation space $(U, R)$ form a lattice with several interesting properties (see e.g. \cite{Po88}, \cite{Com}, \cite{D}, \cite{GW}). The goal of this paper is to examine the algebraic structure of fuzzy rough sets for such favorable properties and to draw a comparison to the case of traditional rough sets.

\section{Preliminary observations}
\label{preliminary}

Let $R$ be a fuzzy relation with a map $\mu_R:U^2 \longrightarrow [0, 1]$. The set $S_{R}:=\{\mu_{R}(x,y)\mid x,y\in U\} \subseteq [0,1] $ is called the \emph{spectrum}
of $R$. We say that a fuzzy relation $R$ has a \emph{dually well-ordered
spectrum}, if any nonempty subset of $S_{R}$ has a maximal element. This is
equivalent to the fact that for any $x\in U$ and any crisp set $B\subseteq U$,
$B\neq\emptyset$ there exists at an element $m_{x}\in B$ such that
\[
\text{sup}\{\mu_{R}(x,y)\mid y\in B\}=\text{max}\{\mu_{R}(x,y)\mid y\in
B\}=\mu_{R}(x,m_{x})\text{.}%
\]

\noindent Observe that this is the case when the spectrum $S_{R}$ of $R$ is a
finite set. If $R$ has a dually well-ordered spectrum, then for any
crisp set $A\subseteq U$, $A\neq\emptyset$%
\[
\mu_{\lbrack A]_{R}}(x)=1-\text{max}\{\mu_{R}(x,y)\mid y\notin A\}\text{,}%
\]%
\[
\mu_{\lbrack A]^{R}}(x)=\text{max}\{\mu_{R}(x,y)\mid y\in A\}\text{.}%
\]

\noindent A similar approach in case of finite (crisp) base sets can be found in \cite{WC}, for decision attributes of decision tables in order to introduce distance measures on fuzzy rough sets. As we pointed out previously the fuzzy rough set corresponding to a
crisp set $A$ is a pair of mappings $(\mu_{\lbrack A]_{R}}$,$\mu_{\lbrack
A]^{R}})$. Let us denote the collection of these pairs by $\mathcal{RS}(U,R)$,
i.e. let
\[
\mathcal{RS}(U,R):=\left\{  \left(  \mu_{\lbrack A]_{R}},\mu_{\lbrack A]^{R}%
}\right)  \mid A\subseteq U\right\}  \text{.}%
\]

\noindent The elements of $\mathcal{RS}(U,R)$ can be ordered by the component-wise
order as follows:

\begin{center}
$\left(  \mu_{\lbrack A]_{R}},\mu_{\lbrack A]^{R}}\right)  \leq\left(
\mu_{\lbrack B]_{R}},\mu_{\lbrack B]^{R}}\right)  \Leftrightarrow$

$\Leftrightarrow\mu_{\lbrack A]_{R}}(x)\leq\mu_{\lbrack B]_{R}}(x)$ and
$\mu_{\lbrack A]^{R}}(x)\leq\mu_{\lbrack B]^{R}}(x)$, for all $x\in U$,
\end{center}

\noindent obtaining a poset $\left(  \mathcal{RS}(U,R),\leq\right)  $ with
least element $(\mathbf{0},\mathbf{0})$ and greatest element $(\mathbf{1}%
,\mathbf{1})$. In other words, this order is a particular case of the product lattice order. We will prove that for any fuzzy equivalence relation $R$ with
a dually well-ordered spectrum, this poset is a complete lattice.

For any number $\alpha\in\lbrack0,1]$, the crisp relation 

\[
R_{\alpha}:=\{(x,y)\in U^{2}\mid\mu_{R}(x,y)\geq\alpha\} 
\]

\noindent is called an $\alpha
$\emph{-section} ($\alpha$\emph{-level}) of the fuzzy relation $R$. If $R$ is
a fuzzy equivalence, then $R_{\alpha}$ is a crisp
equivalence for any  $\alpha\in\lbrack0,1]$. Denote by $E$ the crisp
equivalence $R_{1}$, i.e. let $E:=\{(x,y)\in U^{2}\mid\mu_{R}(x,y)=1\}$. The
$E$-equivalence class of an element $x\in U$ will be denoted by $[x]_{E}$.
Hence%
\[
\lbrack x]_{E}=\{y\in U\mid(x,y)\in E\}=\{y\in U\mid\mu_{R}(x,y)=1\}\text{.}%
\]

\noindent The following lemma is well-known in the literature, see e.g. \cite{Recasens1}:

\begin{lemma}
\label{lemma:lemma1}
For any $y\in\lbrack x]_{E}$ and $z\in U$ we have $\mu_{R}(z,x)=\mu_{R}(z,y)$.
\end{lemma}

Now, let $S$ be the support of the fuzzy equivalence relation $R$ with membership function $\mu_R$, i.e. let

\[
S = \{(x, y)\in U^2 \mid \mu_R (x, y) > 0\},
\]

\noindent and define $S(z)=\{y\in U\mid (z,y)\in S\}$, where $U$ is the universe of $R$ and $z\in U$ is an arbitrary
element. Obviously, the binary relation $S$ is reflexive and
symmetric and $S(z)=\{y\in U\mid \mu _{R}(z,y)>0\}\neq \emptyset $, for
any $z\in U$. \\
Next, assume that $\mathcal{T}$\emph{\ }is a positive t-norm and let  $(x, y) \in S$ and $(y, z) \in S$ for some $x, y, z \in U$. This means that $\mu_R(x, y)>0$ and $\mu_R(y, z)>0$. If $R$ is a $\mathcal{T}$-equivalence we obtain: 
$\mu _{R}(x,z)\geq \mathcal{T}(\mu _{R}(x,y),\mu _{R}(y,z))>0$. Therefore, $\mu_R(x, z) > 0$, from which it follows $(x, z) \in S$, meaning that $S$ is an equivalence relation as well. As before, the $S$-equivalence class of an element $x$ will be denoted by $[x]_S$, and clearly $S(x)=[x]_{S}$.

Using the above defined crisp relations $E\subseteq U\times U$ and $S\subseteq U\times U$, we can
assign (crisp) rough sets to any reference set $A\subseteq U$, by defining
its lower and upper approximation with respect to $E$ or $S$:
\[
A_{E}=\{x\in U\mid\lbrack x]_{E}\subseteq A\},\ A^{E}=\{x\in U\mid\lbrack
x]_{E}\cap A\neq\emptyset\},
\]
\[
A_{S}=\{x\in U\mid\ S(x)\subseteq A\},\ A^{S}=\{x\in U\mid\ S(x)\cap A\neq\emptyset\}.
\]

\begin{lemma}
\label{lemma:lemma2}
For any subset $A\subseteq U$ we have
\begin{enumerate}
\item[(i)] $A^{E}=\{x\in U\mid\mu_{\lbrack A]^{R}}(x)=1\}$,

\item[(ii)] $A_{E}=\{x\in U\mid\mu_{\lbrack A]_{R}}(x)>0\}$,

\item[(iii)] $A^{S}=\{x\in U\mid\mu_{\lbrack A]^{R}}(x)>0\}$,

\item[(iv)] $A_{S}=\{x\in U\mid\mu_{\lbrack A]_{R}}(x)=1\}$.
\end{enumerate}

\noindent In other words, assertions (i) and (ii) in Lemma \ref{lemma:lemma2} mean that $A^{E}$ is equal to the core of
the fuzzy set corresponding to $\mu_{\lbrack A]^{R}}$, whereas $A_{E}$ is equal to
the support of the fuzzy set corresponding to $\mu_{\lbrack A]_{R}}$. Similarly, (iii) and (iv) mean that $A^S$ is equal to the support of the fuzzy set corresponding to $\mu_{\lbrack A]^{R}}$, whereas $A_S$ is equal to the core of the fuzzy set corresponding to $\mu_{\lbrack A]_{R}}$.
\end{lemma}

\begin{proof} (i) If $x\in A^{E}$, then there is a $y\in A$
with $(x,y)\in E$, i.e. $\mu_{R}(x,y)=1$. Hence $\mu_{\lbrack A]^{R}}%
(x)=\ $sup$\{\mu_{R}(x,y)\mid y\in A\}=1$. Conversely, suppose that
$\mu_{\lbrack A]^{R}}(x)=1$ for some $x\in U$. Since $R$ has a dually
well-ordered spectrum, this means that max$\{\mu_{R}(x,y)\mid y\in A\}=1$,
i.e. there exists a $y_{x}\in A$, with $\mu_{R}(x,y_{x})=1$. Then
$(x,y_{x})\in E$, whence $[x]_{E}\cap A\neq\emptyset$. This yields $x\in
A^{E}$.

\smallskip

\noindent(ii) If $x\in A_{E}$, then $[x]_{E}\subseteq A$. This means that
there is no $y\notin A$ with $(x,y)\in E$, i.e. such that $\mu_{R}%
(x,y)=1$. Since $R$ has a dually well-ordered spectrum, the set $\{\mu
_{R}(x,y)\mid y\notin A\}$ has (at least one) maximal element $\mu_{R}%
(x,y_{m})$, where $y_{m}\notin A$. Then $\mu_{R}(x,y_{m})<1$, and we obtain
$\mu_{\lbrack A]_{R}}(x)=1-\ $max$\{\mu_{R}(x,y)\mid y\notin A\}=1-\mu
_{R}(x,y_{m})>0$. Conversely, assume that $\mu_{\lbrack A]_{R}}(x)>0$, for
some $x\in U$. Then for any $y\notin A$ we get \\ $1-\mu_{R}(x,y)\geq$\ inf$\{1-\mu_{R}(x,y)\mid y\notin A\}=$ $\mu_{\lbrack
A]_{R}}(x)>0$. This implies $\mu_{R}(x,y)<1$, for each $y\notin A$. Hence
there is no $y\notin A$ with $\mu_{R}(x,y)=1$, i.e. with $(x,y)\in E$.
This yields $[x]_{E}\subseteq A$. Hence $x\in A_{E}$, and this proves (ii).

\smallskip

\noindent(iii) Assume that $x \in A^S$. This can only be if there exists a $y_m \in A$ with $\mu_R(x, y_m) > 0$. Then $\mu_{[A]^R}(x) = \text{sup}\{ \mu_R(x, y) \mid y \in A \} \ge \mu_R(x, y_m) > 0$, yielding that $x$ is in the support of $\mu_{[A]^R}$. Conversely, assume that $x$ is in the support of $\mu_{[A]^R}$, meaning that $\mu_{[A]^R}(x) = \text{sup}\{ \mu_R(x, y) \mid y \in A \} > 0$. This can only happen if there exists $y_m \in A$ with $\mu_R(x, y_m) > 0$, implying $x \in A^S$.

\smallskip

\noindent(iv) Let $x \in A_S$, i.e. $ S(x) \subseteq A$. Then, by the definition of $S$, for every $y \notin A, \mu_R(x, y) = 0$. Thus, $\mu_{[A]_R}(x) = \text{inf}\{ 1 - \mu_R(x, y) \mid y \notin A \} = 1 - 0 = 1$, i.e. $x$ is in the core of $\mu_{[A]_R}(x)$. The reverse implication yielding $ S(x) \subseteq A $ can be easily checked.
\end{proof}

The assertion of the following proposition is implicitly contained in \cite{Cornelis1}. In fact, it is based on the following observation: \\
\noindent Having two fuzzy (or crisp) equivalence relations $E$ and $R$ with $E \subseteq R$, if we calculate the upper approximation of a fuzzy (or crisp) set $\mu$ for $E$ and then the upper approximation of the obtained fuzzy set for $R$, we get the upper approximation for $\mu$ by $R$. Dually for lower approximations (see also \cite{JKR}).

\begin{proposition}
\label{prop:prop1}
Let $R$ be a fuzzy $\mathcal{T}$-equivalence
on $U$ with a dually well-ordered spectrum and $E:=\{(x,y)\in
U^{2}\mid\mu_{R}(x,y)=1\}$. Then for any set $A\subseteq U$ we
have
\[
\mu_{\lbrack A]^{R}}=\mu_{\lbrack A^{E}]^{R}}\text{ and }\mu_{\lbrack
A]_{R}}=\mu_{\lbrack A_{E}]_{R}}\text{\emph{.}}%
\]
\end{proposition}

\section{Main results}
\label{mainresults}

In what follows, denote as usually by $\left(  \text{RS}(U,E),\leq\right)  $
the lattice of rough sets defined by the equivalence relation $E$.

\begin{theorem}
\label{thm:theorem1}
Let $R$ be a fuzzy $\mathcal{T}$-equivalence on
$U$ with a dually well-ordered spectrum. Then $\left(  \mathcal{RS}%
(U,R),\leq\right)  $ is a complete lattice isomorphic to \break $\left(
\text{RS}(U,E),\leq\right)$.
\end{theorem}

\begin{proof} For each (crisp) rough set $\left(  A_{E}%
,A^{E}\right)  \in\ $RS$(U,E)$ we will assign the fuzzy rough set
corresponding to the crisp set $A$, i.e. the pair $\left(  \mu_{\lbrack
A]_{R}},\mu_{\lbrack A]^{R}}\right)  $. Observe, that the function $f\colon\ $RS$(U,E)\rightarrow\mathcal{RS}(U,R)$,

\medskip

$f\left(  \left(
A_{E},A^{E}\right)  \right)  =\left(  \mu_{\lbrack A]_{R}},\mu_{\lbrack
A]^{R}}\right)  $, where $\left(  A_{E},A^{E}\right)  \in\ $RS$(U,E)$,

\medskip

\noindent is well-defined, because $\left(  A_{E}%
,A^{E}\right)  =\left(  B_{E},B^{E}\right)  $ for some $A,B\subseteq U$
implies $A_{E}=B_{E}$, $A^{E}=B^{E}$, and hence, in view of Proposition \ref{prop:prop1}, we obtain
\[
f\left(  \left(  A_{E},A^{E}\right)  \right)  =\left(  \mu_{\lbrack A]_{R}%
},\mu_{\lbrack A]^{R}}\right)  =\left(  \mu_{\lbrack A_{E}]_{R}},\mu_{\lbrack
A^{E}]^{R}}\right)  =
\]
\[
=\left(  \mu_{\lbrack B_{E}]_{R}},\mu_{\lbrack B^{E}]^{R}}\right)  =\left(
\mu_{\lbrack B]_{R}},\mu_{\lbrack B]^{R}}\right)  =f\left(  \left(
B_{E},B^{E}\right)  \right)  .
\]

\noindent In addition, $f$ is order-preserving because $\left(  A_{E},A^{E}\right)
\leq\left(  B_{E},B^{E}\right)  $ implies $A_{E}\subseteq B_{E}$,
$A^{E}\subseteq B^{E}$, and this yields $\mu_{\lbrack A_{E}]_{R}}\leq
\mu_{\lbrack B_{E}]_{R}}$ and $\mu_{\lbrack A^{E}]^{R}}\leq\mu_{\lbrack
B^{E}]^{R}}$. Thus we obtain:

$f\left(  \left(  A_{E},A^{E}\right)  \right)  =\left(  \mu_{\lbrack
A_{E}]_{R}},\mu_{\lbrack A^{E}]^{R}}\right)  \leq\left(  \mu_{\lbrack
B_{E}]_{R}},\mu_{\lbrack B^{E}]^{R}}\right)  =f\left(  \left(  B_{E}%
,B^{E}\right)  \right)  $.

\noindent Clearly, $f$ is onto, since for any $\left(
\mu_{\lbrack X]_{R}},\mu_{\lbrack X]^{R}}\right)  \in\mathcal{RS}(U,R)$,
$X\subseteq U$ is a crisp set, and hence $f\left(  \left(  X_{E},X^{E}\right)
\right)  =\left(  \mu_{\lbrack X]_{R}},\mu_{\lbrack X]^{R}}\right)  $.
Now, to prove that $f$ is an order-isomorphism, it suffices to show
that $f\left(  \left(  A_{E},A^{E}\right)  \right)  \leq f\left(  \left(
B_{E},B^{E}\right)  \right)  $ implies $\left(  A_{E},A^{E}\right)
\leq\left(  B_{E},B^{E}\right)  $, for any $\left(  A_{E},A^{E}\right)
,\left(  B_{E},B^{E}\right)  \in\ $RS$(U,E)$.

Indeed,  $f\left(  \left(  A_{E},A^{E}\right)  \right)  \leq f\left(  \left(
B_{E},B^{E}\right)  \right)  $ yields that $\left(  \mu_{\lbrack A]_{R}%
}(x),\mu_{\lbrack A]^{R}}(x)\right)  \leq\left(  \mu_{\lbrack B]_{R}}%
(x),\mu_{\lbrack B]^{R}}(x)\right)  $, for all $x\in U$. Hence we get
$\mu_{\lbrack A]_{R}}(x)\leq\mu_{\lbrack B]_{R}}(x)$ and $\mu_{\lbrack A]^{R}%
}(x)\leq\mu_{\lbrack B]^{R}}(x)$, for any $x\in U$. Now, in view of Lemma \ref{lemma:lemma2}, we
obtain:

$A^{E}=\{x\in U\mid\mu_{\lbrack A]^{R}}(x)=1\}\subseteq\{x\in U\mid
\mu_{\lbrack B]^{R}}(x)=1\}=B^{E}$, and

$A_{E}=\{x\in U\mid\mu_{\lbrack A]_{R}}(x)>0\}\subseteq\{x\in U\mid
\mu_{\lbrack B]_{R}}(x)>0\}=B_{E}$.

\noindent Hence $\left(  A_{E},A^{E}\right)  \leq\left(  B_{E},B^{E}\right)
$, and this proves that $f$ is an order-isomorphism. Since $\left(
\text{RS}(U,E),\leq\right)  $ is a complete lattice, we obtain that  $\left(
\mathcal{RS}(U,R),\leq\right)  $ is also a complete lattice isomorphic to
$\left(  \text{RS}(U,E),\leq\right)  $.
\end{proof}

As an immediate consequence, in view of  \cite{Po88} \cite{Com} \cite{GW} we obtain:

\begin{corollary}
\label{cor:corollary1}
If $R$ is a fuzzy equivalence on
the set $U$ with a dually well-ordered spectrum, then $\left(
\mathcal{RS}(U,R),\leq\right)$ is a completely distributive regular
double Stone lattice.
\end{corollary}

\begin{proof} It is known that the rough set lattice $\left(
\text{RS}(U,E),\leq\right)  $ is completely distributive regular double Stone
lattice. Hence Corollary \ref{cor:corollary1} is obtained by applying the isomorphism established
in Theorem \ref{thm:theorem1}.
\end{proof}

\begin{example}

Let the universe be $U = \{a, b, c, d, e\}$ and the fuzzy equivalence relation $R$ be given by Table \ref{tab:fuzzy}. The corresponding $E=\{(x,y)\in U^{2}\mid\mu_{R}(x,y)=1\}$ relation can be seen on Figure \ref{fig:fuzzyE} (loops are not noted for simplicity). Table \ref{tab:lattice} shows the lower and upper approximations of fuzzy relation $R$ and of the (crisp) equivalence relation $E$. Figure \ref{fig:lattice} shows the Hasse-diagram of the lattice $\left(\mathcal{RS}(U,R),\leq\right)$. Here, the nodes are represented as tables, where the top row represents the membership function of the upper approximation of $R$, and the bottom row represents the membership function of the lower approximation of $R$.

\end{example}

\begin{table}[H]
    \centering
    \begin{tabular}{|P{0.5cm}|P{0.5cm}|P{0.5cm}|P{0.5cm}|P{0.5cm}|P{0.5cm}|}
        \hline
         $R$ & $a$ & $b$ & $c$ & $d$ & $e$  \\ \hline
         $a$ & 1 & 1 & 0.5 & 0 & 0 \\ \hline
         $b$ & 1 & 1 & 0.5 & 0 & 0 \\ \hline
         $c$ & 0.5 & 0.5 & 1 & 0 & 0 \\ \hline
         $d$ & 0 & 0 & 0 & 1 & 1 \\ \hline
         $e$ & 0 & 0 & 0 & 1 & 1 \\ \hline
    \end{tabular}
    \caption{An example fuzzy equivalence relation $R$.}
    \label{tab:fuzzy}
\end{table}

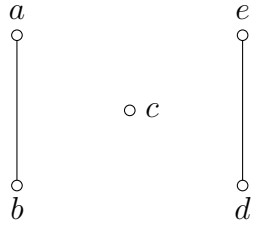
\begin{figure}[H]
    \centering
    \begin{tikzpicture}
    \draw (-1.5, 2.3) node {$a$};
    \draw (-1.5, -0.3) node {$b$};
    \draw (0.3, 1) node {$c$};
    \draw (1.5, -0.3) node {$d$};
    \draw (1.5, 2.3) node {$e$};
    
    \draw (-1.5, 2) -- (-1.5, 0);
    \draw (1.5, 2) -- (1.5, 0);
    
    \draw[fill=white] (-1.5, 0) circle [radius=2pt];
    \draw[fill=white] (1.5, 0) circle [radius=2pt];
    \draw[fill=white] (0, 1) circle [radius=2pt];
    \draw[fill=white] (-1.5, 2) circle [radius=2pt];
    \draw[fill=white] (1.5, 2) circle [radius=2pt];
    \end{tikzpicture}
    \caption{The equivalence relation $E$ corresponding to $R$.}
    \label{fig:fuzzyE}
\end{figure}

\begin{table}[H]
\centering
\resizebox{0.73\textwidth}{!}{%
\begin{scriptsize}
\begin{tabu} to \linewidth {|X[c,m]|X[c,m]|X[c,m]|X[c,m]|X[c,m]|X[c,m]|X[c,m]|X[c,m]|X[c,m]|X[c,m]|X[c,m]|X[c,m]|X[c,m]|}
\hline

\multirow{2}{*}{$A$} & \multirow{2}{*}{$A_E$} & \multirow{2}{*}{$A^E$} & \multicolumn{5}{c|}{$\mu_{[A]_R}(x)$} & \multicolumn{5}{c|}{$\mu_{[A]^R}(x)$} \\ \cline{4-13} 
&&&   a  &  b  &  c  &  d  & e  & a & b & c & d & e \\ \hline
   $\emptyset$    &     $\emptyset$   &    $\emptyset$    & 0 & 0 & 0 & 0 & 0 & 0 & 0 & 0 & 0 & 0  \\
   \hline
a &     $\emptyset$   &        ab         & 0 & 0 & 0 & 0 & 0 &  1  &  1  &  0.5  &  0  & 0 \\
\hline
b &     $\emptyset$   &        ab         & 0 & 0 & 0 & 0 & 0 &  1  &  1  &  0.5  &  0  & 0  \\
\hline
c &         c         &        c          &  0  &  0  & 0.5 &  0  & 0  &  0.5  & 0.5 & 1  & 0  & 0  \\
\hline
d &     $\emptyset$   &        de         &  0 & 0 & 0 & 0 & 0  &  0  & 0 & 0  & 1  & 1  \\ \hline
e &     $\emptyset$   &        de         &  0 & 0 & 0 & 0 & 0  &  0  & 0 & 0  & 1  & 1  \\ \hline
ab &         ab        &        ab         &  0.5  &  0.5  & 0 & 0 & 0  &  1  & 1 & 0.5  & 0  & 0  \\ \hline
ac &         c         &        abc        &  0  &  0  & 0.5 & 0 & 0  &  1  & 1 & 1  & 0  & 0  \\ \hline
ad &    $\emptyset$   &     abde       &  0 & 0 & 0 & 0 & 0  & 1  & 1 & 0.5 & 1 &  1 \\ \hline
ae &    $\emptyset$   &     abde       &   0 & 0 & 0 & 0 & 0 & 1  & 1 & 0.5 & 1 &  1 \\ \hline
bc &         c        &     abc        &  0  &  0  & 0.5 & 0 & 0  &  1  & 1 & 1  & 0  & 0  \\ \hline
bd &    $\emptyset$   &     abde       &   0 & 0 & 0 & 0 & 0 & 1  & 1 & 0.5 & 1 &  1 \\ \hline
be &    $\emptyset$   &     abde       &   0 & 0 & 0 & 0 & 0 & 1  & 1 & 0.5 & 1 &  1 \\ \hline
cd &         c        &     cde        &  0  &  0  & 0.5  & 0 & 0  &  0.5  & 0.5 & 1  & 1  & 1  \\ \hline
ce &         c        &     cde        &  0  &  0  & 0.5  & 0 & 0  &  0.5  & 0.5 & 1  & 1  & 1  \\ \hline
de &        de        &     de         &  0  &  0  & 0    & 1 & 1  &  0  & 0 & 0  & 1  & 1  \\ \hline
abc &       abc       &     abc        &  1  &  1  & 1    & 0 & 0  &  1  & 1 & 1  &  0 &  0 \\ \hline
abd &       ab        &     abde       &  0.5  & 0.5 & 0  & 0 & 0  & 1  & 1 & 0.5 & 1 &  1 \\ \hline
abe &       ab        &     abde       & 0.5   & 0.5 & 0  & 0 & 0  & 1  & 1 & 0.5 & 1 &  1 \\ \hline
acd &       c         &      U         & 0   & 0  & 0.5   & 0 & 0  & 1 & 1 & 1 & 1 & 1   \\ \hline
ace &       c         &      U         & 0   & 0   & 0.5  & 0 & 0  & 1 & 1 & 1 & 1 & 1   \\ \hline
ade &       de        &     abde       & 0   &  0  & 0    & 1 & 1  & 1  & 1 & 0.5 & 1 &  1 \\ \hline
bcd &       c         &       U        & 0   &  0  & 0.5  & 0 & 0  &  1 & 1 & 1 & 1 & 1   \\ \hline
bce &       c         &       U        &  0  &  0  & 0.5  & 0 & 0  &  1 & 1 & 1 & 1 & 1   \\ \hline
bde &       de         &    abde        &  0  &  0  & 0   & 1 & 1  & 1  & 1 & 0.5 & 1 &  1 \\ \hline
cde &       cde        &      cde       &  0  &  0  & 0.5 & 1 & 1  & 0.5  & 0.5 & 1  & 1  & 1  \\ \hline
abcd &      abc         &        U       &  1  &  1  & 1  & 0 & 0  & 1 & 1 & 1 & 1 & 1   \\ \hline
abce &      abc         &        U       &  1  &  1  & 1  & 0 & 0  & 1 & 1 & 1 & 1 & 1   \\ \hline
abde &      abde        &      abde      &  0.5 & 0.5 & 0 & 1 & 1  & 1  & 1 & 0.5 & 1 &  1 \\ \hline
acde &      cde         &      U         &  0  & 0  & 0.5 & 1 & 1  & 1 & 1 & 1 & 1 & 1   \\ \hline
bcde &      cde         &      U         &  0  & 0  & 0.5 & 1 & 1  & 1 & 1 & 1 & 1 & 1   \\ \hline
U &         U        &       U        &    1 & 1 & 1 & 1 & 1  & 1 & 1 & 1 & 1 & 1   \\ \hline

\end{tabu}
\end{scriptsize}
}
\bigskip
\caption{Approximations on $U$ given by the equivalence $E$ and fuzzy relation $R$. }
\label{tab:lattice}
\end{table}

\bigskip

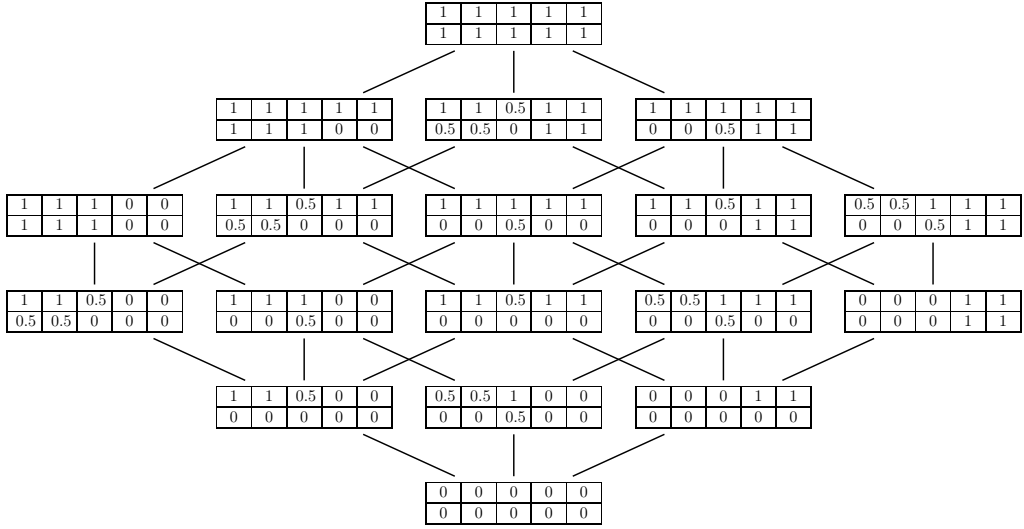
\begin{figure}[H]
    \centering
    \resizebox{\textwidth}{!}{%
    \begin{tikzpicture}[node distance= 5cm,line width=1pt, auto]
    \begin{small}
        \title{}
        \node(min) at (0,0)     {
                \begin{tabularx}{4.2cm}{|C|C|C|C|C|}
                    \hline
                     0 & 0 & 0 & 0 & 0 \\ \hline
                     0 & 0 & 0 & 0 & 0 \\
                     \hline
                \end{tabularx}
        };

        \node(a2)      [above=1cm of min]       {
                \begin{tabularx}{4.2cm}{|C|C|C|C|C|}
                    \hline
                     0.5 & 0.5 & 1 & 0 & 0 \\ \hline
                     0 & 0 & 0.5 & 0 & 0 \\
                     \hline
                \end{tabularx}
        };
        \node(a1)      [left of = a2]       {
                \begin{tabularx}{4.2cm}{|C|C|C|C|C|}
                    \hline
                     1 & 1 & 0.5 & 0 & 0 \\ \hline
                     0 & 0 & 0 & 0 & 0 \\
                     \hline
                \end{tabularx}
        };
        \node(a3)      [right of = a2]       {
                \begin{tabularx}{4.2cm}{|C|C|C|C|C|}
                    \hline
                     0 & 0 & 0 & 1 & 1 \\ \hline
                     0 & 0 & 0 & 0 & 0 \\
                     \hline
                \end{tabularx}
        };
        
        \node(b2)      [above=1cm of a1]       {
                \begin{tabularx}{4.2cm}{|C|C|C|C|C|}
                    \hline
                     1 & 1 & 1 & 0 & 0 \\ \hline
                     0 & 0 & 0.5 & 0 & 0 \\
                     \hline
                \end{tabularx}
        };
        \node(b3)      [above=1cm of a2]       {
                \begin{tabularx}{4.2cm}{|C|C|C|C|C|}
                    \hline
                     1 & 1 & 0.5 & 1 & 1 \\ \hline
                     0 & 0 & 0 & 0 & 0 \\
                     \hline
                \end{tabularx}
        };
        \node(b4)      [above=1cm of a3]       {
                \begin{tabularx}{4.2cm}{|C|C|C|C|C|}
                    \hline
                     0.5 & 0.5 & 1 & 1 & 1 \\ \hline
                     0 & 0 & 0.5 & 0 & 0 \\
                     \hline
                \end{tabularx}
        };
        
        \node(b1)      [left of = b2]       {
                \begin{tabularx}{4.2cm}{|C|C|C|C|C|}
                    \hline
                     1 & 1 & 0.5 & 0 & 0 \\ \hline
                     0.5 & 0.5 & 0 & 0 & 0 \\
                     \hline
                \end{tabularx}
        };
        \node(b5)      [right of = b4]       {
                \begin{tabularx}{4.2cm}{|C|C|C|C|C|}
                    \hline
                     0 & 0 & 0 & 1 & 1 \\ \hline
                     0 & 0 & 0 & 1 & 1 \\
                     \hline
                \end{tabularx}
        };
        
        \node(c1)      [above=1cm of b1]       {
                \begin{tabularx}{4.2cm}{|C|C|C|C|C|}
                    \hline
                     1 & 1 & 1 & 0 & 0 \\ \hline
                     1 & 1 & 1 & 0 & 0 \\
                     \hline
                \end{tabularx}
        };
        \node(c2)      [above=1cm of b2]       {
                \begin{tabularx}{4.2cm}{|C|C|C|C|C|}
                    \hline
                     1 & 1 & 0.5 & 1 & 1 \\ \hline
                     0.5 & 0.5 & 0 & 0 & 0 \\
                     \hline
                \end{tabularx}
        };
        \node(c3)      [above=1cm of b3]       {
                \begin{tabularx}{4.2cm}{|C|C|C|C|C|}
                    \hline
                     1 & 1 & 1 & 1 & 1 \\ \hline
                     0 & 0 & 0.5 & 0 & 0 \\
                     \hline
                \end{tabularx}
        };
        \node(c4)      [above=1cm of b4]       {
                \begin{tabularx}{4.2cm}{|C|C|C|C|C|}
                    \hline
                     1 & 1 & 0.5 & 1 & 1 \\ \hline
                     0 & 0 & 0 & 1 & 1 \\
                     \hline
                \end{tabularx}
        };
        \node(c5)      [above=1cm of b5]       {
                \begin{tabularx}{4.2cm}{|C|C|C|C|C|}
                    \hline
                     0.5 & 0.5 & 1 & 1 & 1 \\ \hline
                     0 & 0 & 0.5 & 1 & 1 \\
                     \hline
                \end{tabularx}
        };
        
        \node(d1)      [above=1cm of c2]       {
                \begin{tabularx}{4.2cm}{|C|C|C|C|C|}
                    \hline
                     1 & 1 & 1 & 1 & 1 \\ \hline
                     1 & 1 & 1 & 0 & 0 \\
                     \hline
                \end{tabularx}
        };
        \node(d2)      [above=1cm of c3]       {
                \begin{tabularx}{4.2cm}{|C|C|C|C|C|}
                    \hline
                     1 & 1 & 0.5 & 1 & 1 \\ \hline
                     0.5 & 0.5 & 0 & 1 & 1 \\
                     \hline
                \end{tabularx}
        };
        \node(d3)      [above=1cm of c4]       {
                \begin{tabularx}{4.2cm}{|C|C|C|C|C|}
                    \hline
                     1 & 1 & 1 & 1 & 1 \\ \hline
                     0 & 0 & 0.5 & 1 & 1 \\
                     \hline
                \end{tabularx}
        };
        
        \node(max)      [above=1cm of d2]       {
                \begin{tabularx}{4.2cm}{|C|C|C|C|C|}
                    \hline
                     1 & 1 & 1 & 1 & 1 \\ \hline
                     1 & 1 & 1 & 1 & 1 \\
                     \hline
                \end{tabularx}
        };
        
        \draw(min) -- (a1);
        \draw(min) -- (a2);
        \draw(min) -- (a3);
        
        \draw(a1) -- (b1);
        \draw(a1) -- (b2);
        \draw(a1) -- (b3);
        
        \draw(a2) -- (b2);
        \draw(a2) -- (b4);
        
        \draw(a3) -- (b3);
        \draw(a3) -- (b4);
        \draw(a3) -- (b5);
        
        \draw(b1) -- (c1);
        \draw(b1) -- (c2);
        
        \draw(b2) -- (c1);
        \draw(b2) -- (c3);
        
        \draw(b3) -- (c2);
        \draw(b3) -- (c3);
        \draw(b3) -- (c4);
        
        \draw(b4) -- (c3);
        \draw(b4) -- (c5);
        
        \draw(b5) -- (c4);
        \draw(b5) -- (c5);
        
        \draw(c1) -- (d1);
        
        \draw(c2) -- (d1);
        \draw(c2) -- (d2);
        
        \draw(c3) -- (d1);
        \draw(c3) -- (d3);
        
        \draw(c4) -- (d2);
        \draw(c4) -- (d3);
        
        \draw(c5) -- (d3);
        
        \draw(d1) -- (max);
        \draw(d2) -- (max);
        \draw(d3) -- (max);
        
    \end{small}
    \end{tikzpicture}
    }
    \caption{Lattice of $\left(\mathcal{RS}(U,R),\leq\right)$ based on Example 1.}
    \label{fig:lattice}
\end{figure}

A rough set is called \emph{exact} if the lower approximation and the upper approximation of the set are equal. This notion can be extended to fuzzy rough sets. A fuzzy rough set defined by the fuzzy equivalence $R$ is \emph{exact} if for every $x \in U$, $\mu_{[A]_R}(x) = \mu_{[A]^R}(x)$ holds, where $U$ is the universe of $R$.

The following proposition describes the relationship between exact fuzzy rough sets and the support of the fuzzy equivalence relation.

\begin{proposition}
\label{prop:prop2}
Let $A$ be a (crisp) subset of $U$. Then
\[ 
 \mu_{[A]_R}(x) = \mu_{[A]^R}(x) \text{ for all } {x \in U} \Leftrightarrow \ A_S = A^S.
\]
\end{proposition}

\begin{proof}
Suppose that $x \in A_S = A^S = A$. This means that
\[ 
\mu_{[A]^R}(x) = sup\{ \mu_R(x, y) \mid y \in A \} = \mu_R(x, x) = 1,
\]

\noindent since $R$ is reflexive.

\noindent Now let us examine the lower approximation. If $y \notin A$, then $y \notin S(x)$ either, because $S(x) \subseteq A$. Since $y \notin S(x)$, according to the definition of $S$, it follows that $\mu_R(x, y) = 0$. This is true for every $y \notin A$, yielding
\[ \mu_{[A]_R}(x) = \textrm{inf}\{ 1 - \mu_R(x, y) \mid y \notin A \} = 1. \]

\noindent So we obtain in this case that $\mu_{[A]_R}(x) = \mu_{[A]^R}(x)=1$.

\noindent Now, let $x \notin A_S = A^S = A$. Then $\mu_R(x, y) = 0$ for each $y \in A$, and we have
\[
\mu_{[A]^R}(x) = \textrm{sup}\{ \mu_R(x, y) \mid y \in A \} = 0,  \textrm{and}
\]
\[
\mu_{[A]_R}(x) = \textrm{inf}\{ 1 - \mu_R(x, y) \mid y \notin A \} = 1 - \mu_R(x, x) = 0.
\]

\noindent Hence in this case we obtain $\mu_{[A]_R}(x) = \mu_{[A]^R}(x)=0$.

\noindent Therefore, we proved that $A_S = A^S$ yields $\  \mu_{[A]_R}(x) = \mu_{[A]^R}(x)$, for all $x \in U$.

\smallskip

\noindent Conversely, assume that $\mu_{[A]_R}(x) = \mu_{[A]^R}(x)$, for all $x \in U$. Let $x \in A$ be arbitrary. Then
\[
\textrm{sup}\{ \mu_R(x, y) \mid y \in A \} = \textrm{inf}\{ 1 - \mu_R(x, y) \mid y \notin A \} = 1,
\]

\noindent because $\mu_R(x, x) = 1$ and $\mu_R(x, y) \le 1$ for all $y \in A$. We conclude $\mu_R(x, y) = 0$, for all $y \notin A$, otherwise the infimum on the right side would be strictly less than $1$.

\noindent Assume $(x, y) \in S$, i.e. $y \in S(x)$. Then, $\mu_R(x, y) > 0$ by the definition of $S$, so $y \notin A$ is not possible. Thus, we get $y \in A$ and this implies $x \in A_S$. Hence, $A = A_S$. Then $A^S = (A_S)^S \subseteq A$ implies $A^S = A = A_S$.
\end{proof}

\begin{example}

Let $U = \{a, b, c, d\}$ and let $R$ be a fuzzy relation given by Table \ref{tab:exact}. Table \ref{tab:approximation_2} shows the four sets for which this relation yields exact sets as lower and upper approximations for $R$ and for $S$ (the support of $R$).
\end{example}

\begin{table}[H]
    \centering
    \begin{tabular}{|P{0.5cm}|P{0.5cm}|P{0.5cm}|P{0.5cm}|P{0.5cm}|}
        \hline
         $R$ & $a$ & $b$ & $c$ & $d$  \\ \hline
         $a$ & 1 & 1 & 0.3 & 0 \\ \hline
         $b$ & 1 & 1 & 0.3 & 0 \\ \hline
         $c$ & 0.3 & 0.3 & 1 & 0 \\ \hline
         $d$ & 0 & 0 & 0 & 1 \\ \hline
    \end{tabular}
    \caption{An example fuzzy equivalence relation $R$.}
    \label{tab:exact}
\end{table}

\begin{table}[H]
\centering
\begin{tabular}{|c|c|c|}
\hline
$A$ & $A_S = A^S$ & $\mu_{[A]_R} = \mu_{[A]^R}$ \\
\hline
$\emptyset$ & $\emptyset$ & $\{(a, 0), (b, 0), (c, 0), (d, 0)\}$ \\
\hline
$\{d\}$ & $\{d\}$ & $\{(a, 0), (b, 0), (c, 0), (d, 1)\}$ \\
\hline
$\{a, b, c\}$ & $\{a, b, c\}$ & $\{(a, 1), (b, 1), (c, 1), (d, 0)\}$ \\
\hline
$\{a, b, c, d\}$ & $\{a, b, c, d\}$ & $\{(a, 1), (b, 1), (c, 1), (d, 1)\}$ \\
\hline
\end{tabular}
\bigskip
\caption{Exact fuzzy sets of relation $R$ from Table \ref{tab:exact}.}
\label{tab:approximation_2}
\end{table}

It can be verified that the containment relationship between a base set $A$ and its fuzzy rough approximations is similar to the containment relationship between the base set and its crisp rough approximations, namely:
\[
\text{core}(\mu_{[A]_R}) \subseteq \text{support}(\mu_{[A]_R}) \subseteq A,
\]
\[
A \subseteq \text{core}(\mu_{[A]^R}) \subseteq \text{support}(\mu_{[A]^R}).
\]

\begin{remark}

An important simple case should also be discussed: when we have imperfect information and we are uncertain about setting up the relation. In this simple case, the relationship between two elements can have three possibilities: the elements are certainly related; the elements are certainly not related; the elements might be related, but we are uncertain. We model this with a fuzzy relation $R$, for which
\[
\mu_R(x, y) = \begin{cases}
      1, & \text{if}\ x\ \text{and}\ y\ \text{are certainly related} \\
      0, & \text{if}\ x\ \text{and}\ y\ \text{are certainly not related} \\
      \tfrac{1}{2}, & \text{if}\ x\ \text{and}\ y\ \text{might be related, but we are uncertain}
    \end{cases}.
\]

\medskip

It can be easily checked that for a crisp set $A \subseteq U$, the membership functions of the lower and upper approximations can be given as follows:

\medskip
\begin{enumerate}
\item[(i)]
$\mu_{[A]_R}(x) = \begin{cases}
      0, & \text{if}\ x \notin A_E \\
      \tfrac{1}{2}, & \text{if}\ x \in A_E \setminus A_S \\
      1, & \text{if}\ x \in A_S
    \end{cases},$
\item[(ii)]
$\mu_{[A]^R}(x) = \begin{cases}
      0, & \text{if}\ x \notin A^S \\
      \tfrac{1}{2}, & \text{if}\ x \in A^S \setminus A^E \\
      1, & \text{if}\ x \in A^E
    \end{cases}.$
\end{enumerate}

\end{remark}

\section{Conclusions and further work}
\label{conclusions}

In this paper, we examined the lattice of fuzzy rough sets corresponding to a fuzzy equivalence relation $R$.
We also investigated the relationship between the core/support of the approximations of a fuzzy rough set and the (crisp) approximations corresponding to the core/support of $R$. We have shown that the lattice of fuzzy rough sets is isomorphic to the lattice of rough sets corresponding to $E$, the core of $R$. We also proved that the membership function of an exact fuzzy set (where $A \subseteq U$ is a crisp set and $\mu_{[A]_R}(x) = \mu_{[A]^R}(x)$ for every $x \in U$) is the same as the characteristic function of a (regular) exact set corresponding to $S$, the support of $R$. 

We can extend the investigation of the $E$-based approximation to the $\alpha$-cut $R$-approximation. The related $R_\alpha$ crisp relation for different $\alpha$-levels can be defined in the following way:
\[
R_\alpha = \{ (x, y) \mid \mu_{R}(x, y) \ge \alpha \}.
\]

It is known that $R_\alpha$ is also an equivalence relation whenever $R$ is a fuzzy equivalence. Then we can give the following result:
\[
A^{R_\alpha} = \{ x \in U \mid \mu_{[A]^R}(x) \ge \alpha \},
\]
\[
A_{R_\alpha} = \{ x \in U \mid \mu_{[A]_R}(x) > 1 - \alpha \}.
\]

As a general case, we would like to extend our results using the framework presented in \cite{Cornelis1}, and verifying the lattice-theoretical properties of the generated fuzzy rough sets.

\section*{Acknowledgement}

The authors would like to thank the area editor and the reviewers for their valuable comments and suggestions.

\bibliographystyle{elsarticle-num} 
\bibliography{rough}

\end{document}